\documentclass[12pt]{amsart}
\usepackage{amscd,amssymb,url}
\numberwithin{equation}{section}
\usepackage{tikz}
\usepackage{tabularx} 
\usepackage{tabu}
\usepackage{anyfontsize}
\usepackage{longtable}
\usepackage[plainpages,urlcolor=blue]{hyperref}

\usepackage[all]{xy}

\theoremstyle{plain}
\newtheorem{thm}[subsection]{Theorem}
\newtheorem{lem}[subsection]{Lemma}
\newtheorem{prop}[subsection]{Proposition}
\newtheorem{cor}[subsection]{Corollary}

\theoremstyle{definition}
\newtheorem{rk}[subsection]{Remark}
\newtheorem{definition}[subsection]{Definition}
\newtheorem{ex}[subsection]{Example}

\numberwithin{equation}{section}
\newcommand{\al}{{\alpha}}
\newcommand{\be}{{\beta}}
\newcommand{\om}{{\omega}}

\newcommand{\SSS}{{\mathcal S}}
\newcommand{\PPP}{{\mathcal P}}

\newcommand{\Z}{\mathbb{Z}}

\newcommand{\C}{\mathbb{C}}
\newcommand{\PP}{\mathbb{P}}

\DeclareMathOperator{\mdr}{mdr}
\DeclareMathOperator{\Hess}{Hess}
\DeclareMathOperator{\tr}{tr}
\DeclareMathOperator{\adj}{adj} 
\DeclareMathOperator{\Jac}{Jac}
\DeclareMathOperator{\Fix}{Fix}

\makeatletter
\newcommand{\thickhline}{%
    \noalign {\ifnum 0=`}\fi \hrule height 1pt
    \futurelet \reserved@a \@xhline
}
\newcolumntype{"}{@{\hskip\tabcolsep\vrule width 1pt\hskip\tabcolsep}}
\makeatother

\begin{document}

\author[Alexandru Dimca]{Alexandru Dimca}
\address{Universit\'e C\^ ote d'Azur, CNRS, LJAD, France and Simion Stoilow Institute of Mathematics,
P.O. Box 1-764, RO-014700 Bucharest, Romania.
}
\email{dimca@unice.fr}

\author[Giovanna Ilardi]{Giovanna Ilardi}
\address{Dipartimento Matematica Ed Applicazioni ``R. Caccioppoli''
Universit\`{a} Degli Studi Di Napoli ``Federico II'' Via Cintia -
Complesso Universitario Di Monte S. Angelo 80126 - Napoli - Italia.}
\email{giovanna.ilardi@unina.it}

\author[Grzegorz Malara]{Grzegorz Malara}
\address{Department of Mathematics,
University of the National Education Commission Krakow,
Podchor\c a\.zych 2,
PL-30-084 Krak\'ow, Poland.}
\email{grzegorzmalara@gmail.com}
\author[Piotr Pokora]{Piotr Pokora}
\address{Department of Mathematics,
University of the National Education Commission Krakow
Podchor\c a\.zych 2,
PL-30-084 Krak\'ow, Poland.}
\email{piotr.pokora@up.krakow.pl}

\title[Free curves by adding osculating conics to a given cubic curve]{Construction of free curves by adding osculating conics to a given cubic curve}

\begin{abstract} 
In the present article we construct new families of free and nearly free curves starting from a plane cubic curve $C$ and adding some of its hyperosculating conics. We present results that involve nodal cubic curves and the Fermat cubic. In addition, we provide new insight into the geometry of the $27$ hyperosculating conics of the Fermat cubic curve using well-chosen group actions.
\end{abstract}
 
\maketitle
 

\section{Introduction}

In a recent paper \cite{DIPS}, the authors have constructed new examples of free and nearly free arrangements consisting of smooth plane curves and (high order) inflectional lines. There seems to be a general pattern that arrangements of (smooth) curves with a high order of contact between them lead to new examples of free curves with beautiful geometry.
In this paper we illustrate this fact by constructing new examples of free and nearly free arrangements consisting of a plane cubic curve, most of the time the Fermat cubic or a nodal cubic, and hyperosculating conics, which are defined in Section $2$. We present in Section $3$ results providing a complete characterization of the free arrangements that can be constructed using a nodal cubic and $k$ hyperosculating conics, with $1\leq k \leq 3$, see Theorem \ref{thmIO}. Then, in Section $4$, we construct many new examples of free and nearly free arrangements using the Fermat cubic curve and the associated $27$ hyperosculating conics. Our computations are supported by \verb}SINGULAR} \cite{Singular}. Moreover, we provide new insight into the geometry of the $27$ hyperosculating conics of the Fermat cubic curve using well-known group actions.
In particular, we show that these $27$ hyperosculating conics can be partitioned into $9$ classes, such that two conics are in the same class if and only if they are tangent to each other at exactly one point, see Proposition \ref{A2}. This result allows us to obtain free curves from the Fermat cubic $C$ by adding $2$ (resp. $3$)  hyperosculating conics that are tangent to each other, see Theorem \ref{thm2con} (resp. Theorem \ref{thm3con}). 

Verification of all symbolic computations performed can be done with the code available at \cite{code}.

\section{Preliminaries}

Let $S= \mathbb{C}[x,y,z]$ be the graded polynomial ring in three variables $x,y,z$ with complex coefficients. Let $C: f=0$ be a reduced curve of degree $d$ in the complex projective plane $\mathbb{P}^2$. We denote by $J_f=(f_x,f_y,f_z)$ the Jacobian ideal of $f$, i.e., the homogeneous ideal in $S$ spanned by the partial derivatives $f_x,f_y,f_z$ of $f.$

Consider the graded $S$-module of Jacobian syzygies of $f$, namely $$AR(f)=\bigg\{(a,b,c)\in S^3 : af_x+bf_y+cf_z=0 \bigg\}.$$
Let $ \mdr(f):=\min\{k : AR(f)_k\neq (0)\}$ denote  the minimal degree of a Jacobian syzygy for $f$.  As  $\mdr(f)=0$ implies that  the curve $C$ is a union  of lines, in this paper we may assume that  $\mdr(f)\geq 1$.

We say that  $C: f=0$ is a $m$-syzygy curve if the $S$-module $AR(f)$ is minimally generated by $m$ homogeneous syzygies, $r_1,r_2,\ldots,r_m$, of degree $d_i=\deg r_i$, ordered such that $$1\leq d_1 \leq d_2 \leq\ldots\leq d_m.$$ 
The multiset $(d_1 ,d_2 ,\ldots,d_m)$ is called the exponents of the plane curve $C$. 

Some of the $m$-syzygy curves have been carefully studied, and we refer to \cite{CDI} and \cite{mdr} for the details. 
\begin{definition}
\label{def:Cfree}
A $2$-syzygy curve $C$ is said to be {\emph{free}}, and in that case we have $d_{1}+d_{2}=d-1$.    
\end{definition}
\begin{rk}
    The name \emph{free curve} in Definition \ref{def:Cfree} follows from the fact that for such a curve the $S$-module $AR(f)$ is a free $S$-module of rank $2$.
\end{rk}
\begin{definition}
A $3$-syzygy curve $C$  is said to be {\emph{nearly free}} exactly when $d_3 =d_2$ and $d_1+d_2=d$.
\end{definition}

In order to check whether a given curve is either free or nearly free we can use the below criteria, see \cite[Corollary 1.2 and Theorem 1.3]{Dmax}. For a given reduced plane curve $C : \, f=0$ of degree $d$ we denote by $\tau(C)$ the total Tjurina number of $C$, i.e., the degree of the Jacobian ideal $J_f$. To simplify notation, let $r := \mdr(f) = \mdr(C)$. In order to decide whether our curve is either free or nearly free, we will use the following results.
\begin{thm}[{\cite[Corollary 1.2]{Dmax}}]
\label{freet}
A reduced plane curve $C$ is free if and only if $r \leq (d-1)/2$ and
\begin{equation*} 
\label{eq:freed}
 r^2-r(d-1)+(d-1)^2=\tau(C).
\end{equation*}
\end{thm}
\begin{thm}[{\cite[Theorem 1.3]{Dmax}}]
\label{nfreet}
A reduced plane curve $C$ is nearly free if and only if 
\begin{equation*} 
\label{eq:nearlyFreeMdr}
 r^2-r(d-1)+(d-1)^2=\tau(C)+1.
\end{equation*}
\end{thm}
\begin{rk}
Note that we can express the exponents for free and nearly free curves using $r$ and $d$. More precisely, if $C$ is free, then $(d_1,d_2)=(r,d-1-r)$. If $C$ is nearly free, then $d_{2}=d_{3}$, so we follow the established convention and list only the first two exponents, namely $(d_{1},d_{2})=(r,d-r)$.
\end{rk}

Let us now recall some facts about inflection points and then about sextactic points of a plane curve $C$.
If $p \in C$ is a smooth point of this curve, and $T_pC$ denotes the projective line tangent to $C$ at $p$, then $p$ is an {\it inflection point} if 
 $(C,T_pC)_p \geq 3$, where  $(C,T_pC)_p$ denotes the intersection multiplicity of the curves $C$ and $T_pC$ at the common point $p$. It is well-known that the inflection points of a curve $C$ are contained in the intersection of the curve with the associated Hessian curve, $H(C)$, with defining polynomial
\begin{equation}
\label{eq2}
h=\det \left(
  \begin{array}{ccccccc}
     f_{xx} & f_{xy}& f_{xz}  \\
     f_{xy} & f_{yy}& f_{yz}\\
     f_{xz} & f_{yz}& f_{zz}  \\
   \end{array}
\right).
\end{equation}

Cayley proved in \cite{Cayley} that for every smooth point $p$ on a curve $C$ of degree $d\geq 3$ there exists a unique conic $Q_p$ with the property that $(C,Q_p)_p \geq 5$, called the \emph{osculating conic} at $p$, see Definition \ref{def:Osculating Conic}. When $(C,Q_p)_p \geq 6$, $p$ is called a \emph{sextactic point}, and $Q_p$ is called the \emph{hyperosculating conic} at $p$.

The set of sextactic points of a plane curve $C$ is obtained by intersecting $C$ with some other curve. In \cite{Cayley1} Cayley computed a polynomial with the property that the sextactic points of $C$ are contained in the intersection of $C$ and its zero set, referred to as the \emph{second Hessian} of $C$, $H_2(C)$. The set of intersection points of $C$ and $H_2(C)$ contains not only sextactic points, but also singular points and higher order inflection points.
\begin{rk}
    The formula that Cayley provided was corrected in \cite{Moe}, to which we refer for the proofs of Cayley Theorem stated above.
\end{rk}
We have the following definitions and notations, making Cayley Theorem true.
\begin{definition}
\label{defSS}

Let $C: f=0$ be a plane curve of degree $d$. Then, its second Hessian $H_2(C)$ is a curve of degree $12d-27$ with the defining polynomial
\begin{multline*}(12d^2-
54d +57)h\Jac(f, 
h, \Omega_h)+ (d-
2)(12d-27)h 
\Jac(f, h, 
\Omega_f) \\
 -20(d-
2)^2 \Jac (f, h, 
\Psi),
\end{multline*}
where for polynomials $f,g,h$ we define

\begin{equation}
\label{eq1}
\Jac(f,g,h)=\det \left(
  \begin{array}{ccccccc}
     f_{x} & f_{y}& f_{z}  \\
      g_{x} & g_{y}& g_{z}\\
     h_{x} & h_{y}& h_{z}  \\
   \end{array}
\right).
\end{equation}Note that only the latter term is a proper Jacobian determinant, with
\begin{equation*}
\label{eq33}
\Psi= -\det \left(
  \begin{array}{ccccccc}
    0 & h_{x} & h_{y}& h_{z}  \\
     h_x & f_{xx} & f_{xy}& f_{xz}\\
    h_y & f_{xy} & f_{yy}& f_{yz}  \\
    h_z & f_{xz} & f_{yz}& f_{zz}\\
   \end{array}
\right).
\end{equation*}
Denote by $\Hess(\cdot)$ the classical Hessian matrix, then in the first two terms
$$\Omega= \tr(\Hess(C)^{\adj} \cdot \Hess (H)),
$$
i.e., it is the scalar product

\begin{equation*}
\label{eq3}
\Omega= {\left(
  \begin{array}{ccccccc}
     f_{yy} f_{zz} - f^2_{yz}  \\
     f_{xx}f_{zz} - f^2_{xz}\\
     f_{xx}f_{yy} -  f^2_{xy} \\
     f_{xy}f_{xz} - f_{xx} f_{yz}\\
     f_{xy}f_{yz} - f_{yy} f_{xz}\\
     f_{xz}f_{yz} - f_{zz} f_{xy}\\
   \end{array}
\right)}^T    \left(
  \begin{array}{ccccccc}
     h_{xx} \\
     h_{yy} \\
     h_{zz}  \\
     2 h_{yz} \\
     2 h_{xz} \\
     2 h_{xy} \\
      \end{array}
      \right),
\end{equation*}
and for a variable  $w \in \{x,y,z\}$, we obtain
\begin{equation*}
\label{eq31}
(\Omega_f)_w= {\left(
  \begin{array}{ccccccc}
    ( f_{yy} f_{zz} - f^2_{yz} )_w \\
    ( f_{xx}f_{zz} - f^2_{xz})_w \\
     (f_{xx}f_{yy} -  f^2_{xy})_w \\
    ( f_{xy}f_{xz} - f_{xx} f_{yz})_w \\
     (f_{xy}f_{yz} - f_{yy} f_{xz})_w \\
     (f_{xz}f_{yz} - f_{zz} f_{xy})_w \\
   \end{array}
\right)}^T    \left(
  \begin{array}{ccccccc}
     h_{xx} \\
     h_{yy} \\
     h_{zz}  \\
     2 h_{yz} \\
     2 h_{xz} \\
     2 h_{xy} \\
      \end{array}
      \right),
\end{equation*}

\begin{equation*}
\label{eq32}
(\Omega_h)_w= {\left(
  \begin{array}{ccccccc}
     f_{yy} f_{zz} - f^2_{yz}  \\
     f_{xx}f_{zz} - f^2_{xz}\\
     f_{xx}f_{yy} -  f^2_{xy} \\
     f_{xy}f_{xz} - f_{xx} f_{yz}\\
     f_{xy}f_{yz} - f_{yy} f_{xz}\\
     f_{xz}f_{yz} - f_{zz} f_{xy}\\
   \end{array}
\right)}^T    \left(
  \begin{array}{ccccccc}
     (h_{xx})_w \\
     (h_{yy})_w \\
     (h_{zz})_w \\
     (2 h_{yz})_w \\
     (2 h_{xz})_w\\
     (2 h_{xy})_w\\
      \end{array}
      \right).
\end{equation*}
\end{definition}
\begin{definition}[{\cite[Section 15]{Cayley}}]
\label{def:Osculating Conic}
Let $C\; : \; f=0$ be a plane curve of degree $d$ and let $\Omega$, $h$ and $\Psi$ be defined as in Definition \ref{defSS}. Define the polynomial $\Lambda = -3 \Omega h + 4\Psi$. If a point $p \in C$ is neither an inflection point nor singular, then the osculating conic $Q_p$ at $p$ is defined as a zero set of the polynomial
\begin{equation}
    D^2_f(p)-\left(\frac{2 D_h(p)}{3h(p)}+\Lambda(p)D_f(p)\right)D_f(p),
\end{equation}
where for any polynomial $f$ and point $p$ we have
\begin{equation*}
    \begin{gathered}
        D_f(p)=xf_x(p)+yf_y(p)+zf_z(p), \\
        D^2_f(p)=x^2f^2_x(p)+2xy(f_xf_y)(p)+y^2f^2_y(p)+2xz(f_xf_z)(p)+2yz(f_yf_z)(p)+z^2f^2_z(p).
\end{gathered}
\end{equation*}
\end{definition}

\section{Nodal cubics and hyperosculating conics}

The first  main result of this article is devoted to irreducible nodal cubics. It is well-known that all nodal cubics are projectively equivalent, and the number of sextactic points on a nodal cubic can be determined by the following theorem \cite[Chapter VI, Theorem 17]{Coolidge}.
\begin{thm}[Coolidge]
\label{Coo}
If an irreducible plane curve $C$ of degree $d$ and geometric genus $g$ has only $n$ nodes and $k$ simple cusps as singularities, and its dual $C^*$ has again only nodes and cusps as singularities, then the number of sextactic
points is 
\begin{equation*}
    3(d^2 - 2n  -3k+6(g-1)).
\end{equation*}
\end{thm}
By the theorem, observe that if $C$ is a nodal cubic, then we have $d=3$, $n=1$ and $g=k=0$, hence $C$ has exactly $3$ sextactic points
and  $3$ hyperosculating conics. Note also that a cubic with a cusp singularity does not admit sextactic points.

We take the nodal cubic given by
\begin{equation}
\label{EE}
E : x^3 + y^{3} - xyz = 0.
\end{equation}
Now we can determine the three sextactic points by direct computations of $E$ and $H_2(E)$, or alternatively as in \cite[Appendix]{S}, and get
$$s_{1} = (1:1:2), \quad s_{2} = ( \omega: \omega^2 :2), \quad s_{3} = (\omega^2: \omega:2),$$
where $\omega = e^{2\pi i/3}$.
Moreover, we can determine the corresponding hyperosculating conics, namely
\begin{itemize}
    \item $Q_{1} \, : \, 21(x^2+y^2 )-22xy-6(x+y)z+z^2 =0,$
    \item $Q_{2} \, : \, 21(\omega x^2 + \omega^2 y^2 )-22xy-6(\omega^2 x+\omega y)z+z^2 =0,$
    \item $Q_{3} \, : \, 21(\omega^2 x^2 +\omega y^2)-22xy-6(\omega x+ \omega^2 y)z+z^2 =0.$
\end{itemize}
We can check via a direct calculation that each pair of conics $Q_{i}, Q_{j}$ with $i\neq j$ intersects along $4$ nodes, and each point $s_{i}$, where curves $E$ and $Q_{i}$ meet, delivers an $A_{11}$-singularity. The $A_{11}$-singularity occurs since at the point $s_{i}$ we have two smooth branches with contact of order $6$.

\begin{rk}
\label{symnodal}
The nodal cubic $E : x^3 + y^{3} - xyz = 0$ has a cyclic group of order $3$ leaving it invariant, namely the group $H$ generated by 
$$w : \,\, (x:y:z) \rightarrow (\omega x: \omega^2 y:z).$$
With $s_1,s_2,s_3$ the sextactic points on $E$, observe that 
$ws_1 = (\omega: \omega^2:2) =s_2 $, $ws_2 = (\omega^2:\omega:2)=s_3$, and $w s_3=(1:1:2)=s_1$, hence the $3$ points form an $H$-orbit. Similarly for the hyperosculating conics $Q_1,Q_2,Q_3,$ $wQ_1=Q_2$, $wQ_2=Q_3$, and $wQ_3=Q_1$, hence the hyperosculating conics also form an $H$-orbit.
Moreover, the group $H$ acts transitively on the set of pairs $(Q_i,Q_j)$ of
hyperosculating conics. This remark simplifies a lot the computations, e.g. in order to check that  each pair of conics $Q_{i}, Q_{j}$ with $i\neq j$ intersects along $4$ nodes, it is enough to verify this property for the pair $(Q_1,Q_2)$. 
\end{rk}

Having this preparation done, we can describe the exponents of the curve obtained from a nodal cubic and any number of hyperosculating conics.

\begin{thm}
\label{thmIO} \

a) If $C$ is an arrangement consisting of a nodal cubic curve and one hyperosculating conic, then  the arrangement is free with  exponents $(2,2)$.

b) If $C$ is an arrangement consisting of a nodal cubic curve and two distinct hyperosculating conics, then  the arrangement is free with  exponents $(3,3)$.

c) If $C$ is an arrangement consisting of a nodal cubic curve and three distinct hyperosculating conics, then  the arrangement is a $4$-syzygy curve with exponents $(5,5,5,5)$.

\end{thm}
\begin{rk}
    If $C$ is the arrangement in Theorem \ref{thmIO} $c)$, then $C$ is a {\it maximal Tjurina curve of type} $(d,r)=(9,5)$ as defined in \cite{maxTjurina}. A maximal Tjurina curve $C'$ is the analog of a free curve when $\mdr(C') \geq \deg(C')/2$. The fact that our curve $C$ is indeed maximal Tjurina follows from \cite[Theorem 3.1]{maxTjurina}.
\end{rk}

\begin{proof}[Proof of Theorem \ref{thmIO}]
To prove a) let us consider the arrangement $C_{i} = \{E,Q_{i}\}$ of degree $d=5$ for $i \in \{1,2,3\}$. Observe that $\tau(C_{i})=12$, since our arrangement has one $A_{1}$- and one $A_{11}$- singularity. 

Using \verb}SINGULAR} we compute $r:=\mdr(C_{i}) =2$, hence 
$$r^{2} - r(d-1) + (d-1)^2 = 2^2 - 2\cdot (5-1) + (5-1)^2 = 12 = \tau(C_{i}),$$
so by Theorem \ref{freet} the arrangement $C_{i}$ is free with exponents $(d_{1},d_{2}) = (r,d-1-r) = (2,2)$ for each $i \in \{1,2,3\}$.

To prove b), let  $C_{i,j} = \{E,Q_{i},Q_{j}\}$ of degree $d=7$ with $i\neq j$. Observe that $\tau(C_{i,j})=27$ since we have exactly $5$ nodes and $2$ singularities of type $A_{11}$. Moreover, using \verb}SINGULAR} we can check that $r:=\mdr(C_{i,j})=3$, therefore, again by Theorem \ref{freet}, we have
$$r^{2} - r(d-1) + (d-1)^2 = 3^2 - 3\cdot(7-1) + (7-1)^2 = 27 = \tau(C_{i,j}),$$
so the arrangement $C_{i,j}$ is free with exponents $(d_{1},d_{2}) = (r,d-1-r) = (3,3)$ for each pair $i \neq j$.

Finally, for c), let  $C_{1,2,3} = \{E,Q_{1},Q_{2},Q_{3}\}$. We can check directly, using \verb}SINGULAR}, that $C_{1,2,3}$ is a $4$-syzygy curve with $(d_{1},d_{2},d_{3},d_{4}) = (5,5,5,5)$, and this completes our proof.
\end{proof}

Nearly free arrangements can also be constructed using a nodal cubic $E$ and other configurations of highly tangent conics to $E$. The following example shows one possible such arrangement.
\begin{ex}
In this example, we return to the nodal cubic 
$E$ from \eqref{EE}. We recall that $s_1=(1:1:2)$ is a  sextactic point for $E$, and  the corresponding hyperosculating conic is given by
$$Q_{1} \, : \, 21(x^2+y^2 )-22xy-6(x+y)z+z^2 =0.$$
At the point $p= (2 :4 : 9)$ the osculating conic to $E$ is
$$Q \, : \, 2961x^2-2664xy+2394y^2-1104xz-321yz+32z^2=0 ,$$
and note that the intersection index at $p$ of curves $E$ and $Q$ is equal to $5$.
Using \verb}SINGULAR}, we can directly check that the conics $Q_{1}$ and $Q$ intersect along $4$ nodes. Consider the arrangement $C=\{E,Q_{1},Q\}$, which has one singularity of type $A_{11}$ at $s_1$, one singularity of type $A_{9}$ at $p$, and six points of type $A_{1}$. This leads to $\tau(C) = 26$, and using \verb}SINGULAR} we can check that $r=\mdr(C) = 3$, hence
$$r^2 - r(d-1) + (d-1)^2 = 3^2 - 3\cdot(7-1) + (7-1)^2 = 27 = \tau(C)+1,$$
so by Theorem \ref{nfreet} our curve $C$ is nearly free.
\end{ex}

\section{Smooth cubics and hyperosculating conics}
In this section we focus on free arrangements consisting of a single smooth cubic $E$ and its hyperosculating conics. The first result holds for any smooth cubic $E$, and it shows that if all the singularities of the arrangement not on $E$ are nodes, then the arrangement is never free, 
see Proposition \ref{PP1}. The remaining results in this section concern the arrangements obtained from the smooth Fermat cubic by adding some of its hyperosculating conics.

We know by  Coolidge theorem that every smooth cubic curve has exactly 
$ 3 \cdot 9 = 27$
sextactic points, and therefore we have exactly $27$ hyperosculating conics.
An arrangement of curves consisting of a smooth cubic and $k \in \{1,2, ..., 27\}$ hyperosculating conics will intersect at the sextactic points of the smooth cubic and at the intersection points of the conics. The intersections at the sextactic points will, by B\'ezout's theorem, constitute $k$ singularities of type $A_{11}$. The intersection points of the conics need to be considered more carefully. Again, by B\'ezout's theorem we know that the sum of the intersection multiplicities at these points is $4\cdot \binom{k}{2} = 2k(k-1)$, but there are different possible configurations of singularities.

If we take exactly $k$ hyperosculating conics, then the freeness is governed by the $2k(k-1)$ intersections among the conics. Let us consider the first situation that can potentially occur, namely when $k$ hyperosculating conics are intersecting only along nodes.
\begin{prop}
\label{PP1}
Let $\mathcal{EC}_{k}$ be an arrangement consisting of an elliptic curve $E$ and $k \geq 1$ hyperosculating conics, and assume that the hyperosculating conics are intersecting only along $2k(k-1)$ nodes. Then $\mathcal{EC}_{k}$ is never free. In particular, when $k=1$ the arrangement $\mathcal{EC}_{1}$ is nearly free with exponents $(2,3)$.
\end{prop}

\begin{rk}
\label{rkPP1}
Note that Proposition \ref{A2} on page \pageref{A2} implies that the assumption that the conics intersect only in $2k(k-1)$ nodes holds if and only if $k\leq 9$ and the conics are chosen in distinct sets $P_j$, sets which are defined in the proof of Proposition \ref{A2}.
\end{rk}

\begin{proof}

Let $\mathcal{EC}_{k} \, : \, f_{k}=0$ be the defining equation of degree $\deg(f_k)=2k+3$, with $k\geq 1$ and $r = \mdr(f_{k}) > 0$. 
We discuss first the case $k=1$. Then $\mathcal{EC}_{1}$ has one $A_{11}$-singularity and $\tau(\mathcal{EC}_{1})=11$. On one hand, by \cite{mdr}, $r \geq 2=\mdr(e)$ where $e=0$ is the defining equation of $E$. On the other hand, if $r \geq 3$, then by \cite[Theorem 3.2]{duPCTC} we have the inequality
$$\tau(\mathcal{EC}_{1}) \leq 3^{2} - 3\cdot(5-1)+(5-1)^2-\binom{3}{2}=10,$$
which is a contradiction. It means that $r=2$ and
$$r^2-r(d-1)+(d-1)^2=2^2-2\cdot (5-1)+(5-1)^2=12=\tau(\mathcal{EC}_{1})+1,$$
hence by Theorem \ref{nfreet} arrangement $\mathcal{EC}_{1}$ is nearly free with exponents $(2,3)$.

Now we treat the case $k\geq 2$. We have exactly $k$ singularities of type $A_{11}$ and $2k(k-1)$ nodes, so the total Tjurina number of $\mathcal{EC}_{k}$ is equal to
$$\tau(\mathcal{EC}_{k}) = 2k(k-1) + 11k.$$
Assume that $\mathcal{EC}_{k}$ is free. Then we have
$$(2k+2)^2  -r(2k+2) + r^{2} = \tau(\mathcal{EC}_{k}) = 2k(k-1) + 11k.$$
This is a quadratic polynomial in $r$, 
$$r^{2} -r(2k+2) + 2k^2-k+4 =0,$$
with discriminant 
$$\triangle_{r} = -4k^2 +12k-12,$$
which is strictly negative for $k \geq 2$. Since by assumption $r \in \mathbb{N}$, this is a contradiction.

\end{proof}
Now we  pass to the concrete situation where  we study the Fermat cubic curve $F: x^3+y^3+z^3=0$ and its hyperosculating conics. There are several groups acting on $\PP^2$ and leaving the curve $F$ invariant. Firstly, there is the group $G=\Z_3\times \Z_3$, where the two generators $g_1$ and $g_2$ act by
\begin{equation}
\label{eqA0}
g_1(x:y:z)=(\om x : y:z) \text{ and } g_2(x:y:z)=(x : \om y:z).
\end{equation}
Then there is the symmetric group $\SSS_3$ acting by permutation of coordinates. We denote by $\gamma \in \SSS_3$ a 3-cycle and by $\langle \gamma \rangle$ the cyclic subgroup of order 3 generated by $\gamma$.
Then we put together these actions by considering the semi-direct products $G' \subset G''$, where
$$G' \cong G \rtimes \langle \gamma \rangle \text{ and } G''\cong G \rtimes \SSS_3.$$
An element $t$ in these groups of automorphisms acts on points in $\PP^2$ in the obvious way
and on polynomials in $S$ by composition with $t^{-1}$.
Using these actions we can prove the following.
\begin{prop}
\label{A1}
The $27$ sextactic points of the Fermat cubic curve
correspond exactly to the $G'$-orbit of the point
$$s_1=(1:1:-\al),$$
where $\al=\sqrt[3]{2}$.
\end{prop}
\proof
If the symmetric group $\SSS_3$ acts in a free way on the set of sextactic points of $F$, we get a contradiction,  since  $|\SSS _3| = 6$ does not divide $27$, the number of sextactic points.

It follows that some sextactic points of $F$ are fixed under some permutation.
Note that the only points $(1:a:b)$, with $1 \ne a \ne b \ne 1$  and  $ab \ne 0$, that  are
fixed by some permutation in $\SSS_3$ satisfy either $a^3=b^3=1$ or $a^2=b^2=1$, but the curve $F$ does not contain any of these points. Moreover,  the points of the form $(0:a:b)$ are fixed by some
transposition, but only when either $a=1$ and   $b=-1$ or $a=-1$ and $b=1$. These  points are on $F$, but it can be checked that they are not sextactic points. It follows that the coordinates of a sextactic point are all non-zero, and two of them are equal. Using our permutation group, we may assume that the equal coordinates are on the first two positions, hence $s=(a:a:c)$. By dividing by $a \ne 0$ and by asking that $s \in F$, we see that 
$s=(1:1:\al_0)$ with $\al_0^3=-2$. Finally, using the group $G$ we can arrange that $ \al_0=-\sqrt[3]{2}$.

It is clear that the orbit of such an element under the group $G'$ has exactly 27 elements. An alternative proof of this result can be obtained using the formula for the second Hessian of $F$, namely
$$H_2(F)=(x^3-y^3)(y^3-z^3)(x^3-z^3),$$
see \cite{Sz} for explanations.
\endproof

\begin{ex}
\label{exA1}
Using \cite[Theorem 2.1]{Moe}, that goes back to Cayley, we see that the hyperosculating conic corresponding to the point
$s_1=(1:1:-\al)$ has the equation
\begin{equation}
\label{eqA1}
Q_{s_1} : \, q(1:1:-\al)=(x-y)^2-z(\al^2x +\al^2 y +2\al z)=0,
\end{equation}
for $  \al=\sqrt[3]{2}$.
To get the equations for {\it all the other hyperosculating conics}, we apply the elements of the group $G'$ to Equation \ref{eqA1}. For instance, to get the equation of the conic $Q_{s_{2}}$ corresponding to the sextactic point $s_2=(1:-\al:1)$ we use the $2$-cycle 
$$\gamma(x:y:z) = (x:z:y)$$  
and get $s_2=\gamma (s_1)$. Then
\begin{align}
\label{eqA11}
    \begin{aligned}
        Q_{s_2} : \, q(1:-\al:1)(x,y,z)&=q(1:1:-\al)\cdot \gamma^{-1}(x,y,z)\\
        &= q(1:1:-\al)(x,z,y)\\
        &=(x-z)^2-y(\al^2x +\al^2 z +2\al y)=0, 
    \end{aligned}
\end{align}
since $\gamma^{-1}=\gamma$ in this case.
Similarly, to get the equation of the conic $Q_{s_{3}}$ corresponding to the sextactic point $s_3=(1:\om:-\al)$,
we use the transformation $g_2$ defined in \eqref{eqA0}
which satisfies $s_3=g_2(s_1)$. This yields
\begin{align}
\label{eqA111}
    \begin{aligned}
        Q_{s_3}: q(1:\om:-\al)(x,y,z)&=q(1:1:-\al)\cdot g_2^{-1}(x,y,z)\\
        &=q(1:1:-\al)(x,\omega^2 y,z)\\
        &=(x-\om^2y)^2-z(\al^2x +\al^2 \om^2 y +2\al z)=0.
    \end{aligned}
\end{align}
\end{ex}
Our considerations from Example \ref{exA1} lead us to the following corollary.
\begin{cor}
\label{corA1}
The group $G'$ acts freely on the set of hyperosculating conics and on the set of pairs of hyperosculating conics. In particular, the number of $G'$-orbits on the set of pairs of hyperosculating conics is equal to 13.
\end{cor}
\proof
The fact that the group $G'$ acts freely on the set of hyperosculating conics follows from Proposition \ref{A1} since this set is just a $G'$-orbit of order $|G'|=27$. To prove the claim for the pairs, let
$\{Q,Q'\}$ be a pair of hyperosculating conics and $t \in G'$ such that
$\{Q,Q'\}=\{tQ,tQ'\}$. If $Q=tQ$ it follows from the first part that $t=e$. Assume now that $Q'=tQ$ and $Q=tQ'$. It follows that
$t^2Q=Q$ and hence $t^2=e$. But the group $G'$ has odd order $27$, and hence it has no element of order $2$. The last claim follows from the fact that the number of  pairs of hyperosculating conics is
$$\binom{27}{2}=27 \cdot 13.$$
\endproof

\medskip
In the sequel, we need the following elementary result.
\begin{lem}
\label{L1}
Let $L_1:  \ell_1=0$ and $L_2: \ell_2=0$ be two distinct lines in $\PP^2$, and let $p=L_1 \cap L_2$. Then any smooth conic which is tangent to
$L_1$ at the point $p$ has an equation of the form
$$Q: \ell_1 \ell +\ell_2^2=0,$$
where $\ell$ is a linear form in $S$.
\end{lem}
\proof
Without loss of generality, let $L_1:x=0$ and $L_2:y=0$, and then $p=(0:0:1)$.
A general conic is given by
$$Q: ax^2+by^2+cz^2+dxy+exz+fyz=0,$$
where $a,b,c,d,e,f \in \C$. Since $p \in Q$, we have $c=0$. Since the tangent at $p$ is $L_1$, it follows that $f=0$. Thus
$$Q: x(ax+dy+ez) +by^2=0.$$
Since $Q$ is smooth, it follows that $b \ne 0$, hence we can take $b=1$.
\endproof
Having this lemma, we can present our next result.
\begin{prop}
\label{A2}
The $27$  hyperosculating conics of the Fermat cubic 
can be partitioned into $9$ sets  $P_j$ for $j\in\{1, \ldots, 9\}$, each set $P_j$ containing $3$ conics, such that the following holds: 
\begin{enumerate}
\item[a)] If two conics $Q$ and $Q'$ belong to the same set $P_j$, then
the intersection $Q\cap Q'$ consists of two nodes $A_1$ and one tacnode $A_3$.
\item[b)] If two conics $Q$ and $Q'$ do not belong to the same set $P_j$, then
the intersection $Q\cap Q'$ consists of four nodes $A_1$.

\end{enumerate}
\end{prop}
\proof
Note that the equations of the three conics $Q_{s_1}$ given in \eqref{eqA1}, corresponding to each value of $-\alpha$, all have the form
$$(x-y)^2+z\ell=0,$$
for various linear forms $\ell$. Using Lemma \ref{L1}, with $L_1: z=0$ and $L_2: x-y=0$, it follows that these 3 conics pass through the point 
$$p_1=(1:1:0)$$
and are tangent to the line $z=0$. Let $P_1$ denote set of hyperosculating conics that vanish at $p_{1}$. Observe that $P_{1}$ is the set of the form $\{Q_{s_1}  \  |  \  \alpha^3=2\}$, in other words this is the set of all hyperosculating conics which vanish at $p_1$. 
Indeed, such a conic must be in the ideal generated by $(x-y)$ and $z$,
and, with this remark in hand, Example \ref{exA1} proves this claim.
The orbit of the point $p_1$ under the action of $G'$ or $G''$ consists of the following $9$ points:
\begin{gather*}
    \begin{gathered}
        p_1=(1:1:0), \; p_2=(1:\om:0), \; p_3=(1:\om ^2:0),\\
        p_4=(0:1:1),  \; p_5=(0:1:\om), \; p_6=(0:1:\om ^2),\\
        p_7=(1:0:1),  \; p_8=(\om:0:1), \; p_9=(\om ^2:0:1),
    \end{gathered}
\end{gather*}
where $\om^3=1$, $\om \ne 1$. Similarly, let $P_l$ for the remaining values of $j$ be the set of all the hyperosculating conics which vanish at $p_l$. Since the points $p_l$ are in one $G'$-orbit, the sets $P_l$ are also in one $G'$-orbit, due to their definition. Hence, to prove claim $a)$, it is enough to check the intersection $Q \cap Q'$ only when $Q,Q' \in P_1$.
And this can be done by using the beginning of our proof or by a direct check with \verb}SINGULAR}.
With this choice, the claim $a)$ is proved.
For claim $b)$, using Corollary \ref{corA1}, it is enough to check for one pair in each of the $13$ orbits. But among these $13$ orbits there is the orbit of the pairs of conics belonging to the same set $P_l$. Hence we have to check only $12$ pairs. This was done by a direct check with \verb}SINGULAR}, see \cite{code}.
\endproof

\begin{rk}
\label{rkA2}
Consider the action of the group $G'$ on the set $\PPP=\{p_1, \ldots, p_9\}$. Then the isotropy group $\Fix(p_1)$ is the cyclic group of order
3 generated by 
$$t=g_1g_2:\,\, (x:y:z) \mapsto (\om x: \om y:z)=(x:y: \om^2z).$$
Note that this group cyclically permutes the three conics
$Q_{s_1}$ that pass through the point $p_1$. Indeed, one has
$$q(1:1: -\be)\cdot t^{-1}(x,y,z)=q(1:1: -\be)(\om^2x , \om^2y ,z)$$
$$=\om(x-y)^2-\om z'((\be \om)^2x +(\be \om)^2y + (\be \om)z')=0,$$
where we set $z'=z/\om=z\om^2$. Hence after division by $\om$ we get the equation for the curve $Q$, corresponding to the sextactic point $$(1:1:-\be\om).$$ The same property holds for any isotropy group $\Fix(p_j)$, namely it cyclically permutes the $3$ conics in $P_j$. This shows, in particular, that the $3$ pairs formed with the conics in $P_j$ are in the same $G'$-orbit.
\end{rk}

Recall that the situation where we have an arrangement consisting of  any smooth cubic and exactly one 
hyperosculating conic was considered in Proposition \ref{PP1}.
Now we consider the case of the Fermat cubic with several 
hyperosculating conics.

\begin{thm}
\label{thm2con}
Let $\mathcal{EC}_{2}$ be an arrangement consisting of the Fermat cubic curve $F$ and two hyperosculating conics $Q$ and $Q'$.
\begin{enumerate}
\item[a)] The two conics $Q$ and $Q'$ belong to the same set $P_j$ if and only if
the arrangement $\mathcal{EC}_{2}$ is free with exponents $(3,3)$.
\item[b)] The two conics $Q$ and $Q'$ do not belong to the same set $P_j$ if and only if
the arrangement $\mathcal{EC}_{2}$ is nearly free with  exponents $(3,4).$

\end{enumerate}

\end{thm}
\begin{proof}
As in the previous case, our proof is based on \verb}SINGULAR}. We start with a general observation regarding the value of ${\rm mdr}$ for our arrangements. Using our code in \verb}SINGULAR}, we check that in both cases $a)$ and $b)$ one has
$$ r=\mdr(\mathcal{EC}_{2})=3.$$
In case $a)$, it is enough to check this property for only one pair, since all pairs in $a)$ are in the same $G$-orbit. In case $b)$, however, we must check the $12$ cases mentioned at the end of the proof of Proposition \ref{A2}.
In both cases our arrangement has degree $d=7$, but we have different singularities. In case a), our arrangement have two singularities of type $A_{11}$, one singularity of type $A_{3}$, and two singularities of type $A_{1}$, hence $\tau(\mathcal{EC}_{2})=27$. Using Theorem \ref{freet}, we can verify that
$$r^2-r(d-1)+(d-1)^2= 3^2 - 3\cdot(7-1) + (7-1)^2 =27= \tau(\mathcal{EC}_{2}),$$
hence $\mathcal{EC}_{2}$ is free with exponents $(d_{1},d_{2}) = (r, d-1-r) = (3,3)$.

In case $b)$, our arrangement has two singularities of type $A_{11}$ and  four singularities of type $A_{1}$, hence $\tau(\mathcal{EC}_{2})=26$. We use Theorem \ref{nfreet}, namely
$$r^2-r(d-1)+(d-1)^2 =  3^2 - 3\cdot(7-1) + (7-1)^2 = 27= \tau(\mathcal{EC}_{2})+1,$$
which shows that this curve is nearly free with exponents $(d_{1},d_{2}) = (r,d-r) = (3,4)$.
\end{proof}

If we add more hyperosculating conics, i.e., for higher values of $k$, we need to determine all singular points of the conic arrangement consisting of other subconfigurations of $27$ hyperosculating conics. It turns out that this is a very elaborate task. For $k=3$ there are a vast number $\binom{27}{3}$ of different cases to check. To get a glimpse of the possibilities and difficulties that occur for higher values of $k$ we round off this paper by studying one type of arrangements of curves consisting of the Fermat cubic curve and three hyperosculating conics.
\begin{thm}
\label{thm3con}
Let $\mathcal{EC}_{3}$ be an arrangement consisting of the Fermat cubic curve and exactly three hyperosculating conics $Q$, $Q' $ and $Q''$, which belong to the same set $P_j$.
Then the arrangement $\mathcal{EC}_{3}$ is free with exponents $(3,5)$.

\end{thm}
\begin{proof}  Note that in view of the $G'$-action, it is enough to perform computations only for the triple of conics in $P_1$.

First observe that, since the arrangement consists of the Fermat cubic curve and three hyperosculating conics, we have $d=9$ and exactly $3$ singularities of type $A_{11}$.   Locally at the common intersection point $p$ of the 3 conics $Q,Q'$ and $Q''$, the singularity obtained is a merger of $3$ tacnodes, i.e., each pair of conics intersects at that point as a tacnode, and we have exactly $3$ such pairs. Using \verb}SINGULAR} we can find the normal form of this singularity which is 
$$g(u,v) = u^3+u^2 v^2 + uv^5 + v^6 .$$
Based on the above data, we verify that our singular point is a ${\rm J}_{2,0}$ singularity according to Arnold's classification of singularities of corank $2$ and $3$-jet equal to a perfect cube, see \cite[page 248]{Arnold}.
Furthermore, we can check that its local Tjurina number is equal to the local Milnor number, i.e., $\tau_{p} (\mathcal{EC}_{3})= \mu_{p}(\mathcal{EC}_{3})=10$.
Moreover, we have additionally $6$ singularities of type $A_{1}$ as the remaining intersections of $Q, Q', Q''$.
Let $f_3=0$ be the defining equation of $\mathcal{EC}_{3}$. By the above description, in total,  $\tau(\mathcal{EC}_{3})=49$, and using \verb}SINGULAR} we compute $r=\mdr(f_3)=3$. Using Theorem \ref{freet}, we have
$$r^2-r(d-1)+(d-1)^2=  3^{2} - 3\cdot(9-1) + (9-1)^2 = 49 = \tau(\mathcal{EC}_{3}),$$
hence $\mathcal{EC}_{3}$ is free with exponents $(d_{1}, d_{2}) = (r,d-1-r) = (3,5)$. 

\end{proof}

\section*{Conflict of Interests}
We declare that there is no conflict of interest regarding the publication of this paper.
\section*{Acknowledgments}
We would like to thank anonymous referees for many very useful comments
that allowed us to improve the paper.

Piotr Pokora is partially supported by The Excellent Small Working Groups Programme \textbf{DNWZ.711/IDUB/ESWG/2023/01/00002} at the University of the National Education Commission Krakow.

Giovanna Ilardi was partially supported by GNSAGA-INDAM.

\end{document}